\documentclass[12pt]{article}
\usepackage{amssymb}
\usepackage{amsfonts}
\usepackage{amsmath}

\newtheorem{theorem}{Theorem}[section]

\newtheorem{conjecture}[theorem]{Conjecture}
\newtheorem{corollary}[theorem]{Corollary}

\newtheorem{definition}[theorem]{Definition}

\newtheorem{lemma}[theorem]{Lemma}

\newenvironment{proof}[1][Proof]{\noindent\textbf{#1.} }{\ \rule{0.5em}{0.5em}}
\input{tcilatex}
\begin{document}

\title{Matrix group actions on product of spheres and Zimmer's program}
\author{Shengkui Ye \\
Xi'an Jiaotong-Liverpool University}
\maketitle

\begin{abstract}
Let $\mathrm{SL}_{n}(\mathbb{Z})$ be the special linear group over integers
and $M^{r}=S^{r_{1}}\times S^{r_{2}},$ $T^{r_{1}}\times S^{r_{2}}$ , or $%
T^{r_{0}}\times S^{r_{1}}\times S^{r_{2}},$ products of spheres and tori. We
prove that any group action of $\mathrm{SL}_{n}(\mathbb{Z})$ on $M^{r}$ by
diffeomorphims or piecewise linear homeomorphisms is trivial if $r<n-1.$
This confirms a conjecture on Zimmer's program for these manifolds.
\end{abstract}

\section{Introduction}

The special linear group $\mathrm{SL}_{n}(\mathbb{Z})$ acts on the Euclidean
space $\mathbb{R}^{n}$ by linear transformations. There is an induced action
on the sphere $S^{n-1}$ ($\subseteq \mathbb{R}^{n}$) by $x\longmapsto
Ax/\Vert x\Vert $ for $x\in S^{n-1}$ and $A\in \mathrm{SL}_{n}(\mathbb{Z})$.
It is believed that this action is minimal in the following sense.

\begin{conjecture}
\label{conj}Any action of a finite-index subgroup of $\mathrm{SL}_{n}(%
\mathbb{Z})$ $(n\geq 3)$ on a compact $r$-manifold by diffeomorphisms
factors through a finite group action if $r<n-1.$
\end{conjecture}

This conjecture was formulated by Farb and Shalen \cite{fs}, which is
related to Zimmer's program \cite{z3}. Conjecture \ref{conj} could be a
special case of a general conjecture in Zimmer's program, in which the
special linear group is replaced by an arbitrary irreducible lattice $\Gamma 
$ in a semisimple Lie group $G$ of $\mathbb{R}$-rank at least 2, and the
integer $n$ is replaced by a suitable integer $h(G).$ Some forms of
Conjecture \ref{conj} has been discussed by Weinberger \cite{sw}. These
conjectures are part of a program to generalize the Margulis Superrigidity
Theorem to a nonlinear context.

In general, it is difficult to prove this conjecture. Usually, we have to
assume either that the group action preserves additional geometric
structures or that the lattices themselves are special. The following is an
incomplete list of some results in this direction. For more details of
Zimmer's program and related topics, see survey articles of Zimmer and
Morris \cite{zm}, Fisler \cite{fi} and Labourie \cite{la}.

When $r=1$ and $M=S^{1}$, Ghys \cite{gh} and Burger-Monod \cite{bm} show
that every $C^{1}$ action of a lattice with $\mathbb{R}$-rank $\geq 2$ on $%
S^{1}$ factors through a finite group action. Witte \cite{wi} proves that
Conjecture \ref{conj} is true for an arithmetic lattice $\Gamma $ with $%
\mathbb{Q}$-rank$(\Gamma )\geq 2$. He actually proves that the topological
version is true as well.

When $r=2,$ the group action is smooth real-analytic and $M$ is a compact
surface other than the torus or Klein bottle, Farb and Shalen \cite{fs}
prove that Conjecture \ref{conj} is true for $n\geq 5$ (more generally for $%
2 $-big lattices). When the group action is smooth real-analytic and
volume-preserving, they also show that this result could be extended to all
compact surfaces.

Polterovich (see Corollary 1.1.D of \cite{Po}) proves that if $n\geq 3$,
then any action by $\mathrm{SL}(n,\mathbb{Z})$ on a closed surface other
than the sphere $S^{2}$ and the torus $T^{2}$ by area preserving
diffeomorphisms factors through a finite group action. When $r=2$ and the
group action is by area preserving diffeomorphisms, Franks and Handel \cite%
{fh} prove that Conjecture \ref{conj} is true for an almost simple group
containing a subgroup isomorphic to the three-dimensional integer Heisenberg
group (eg. any finite-index subgroup of $\mathrm{SL}_{n}(\mathbb{Z})$ for $%
n\geq 3$). When $M$ is a compact surface with nonempty boundary, Ye \cite%
{ye1} proves Conjecture \ref{conj} is true for $\mathrm{SL}_{n}(\mathbb{Z})$
($n\geq 5).$ The conjecture is still open for general smooth actions on
closed $2$-dimensional manifolds.

Bridson and Vogtmann \cite{bv} prove Conjecture \ref{conj} for $\mathrm{SL}%
_{n}(\mathbb{Z})$ and $M=S^{r},$ the sphere of dimension $r.$ Weinberger 
\cite{we} shows that when $r<n,$ any group action of $\mathrm{SL}_{n}(%
\mathbb{Z})$ on the lower dimensional torus $T^{r}$ by diffeomorphisms is
trivial, i.e., Conjecture \ref{conj} holds for the group $\mathrm{SL}_{n}(%
\mathbb{Z})$ itself and $M=T^{r}$. The aim of this article is to provide
more manifolds satisfying Conjecture \ref{conj}.

Our first result is to prove Conjecture \ref{conj} for manifolds with few
Betti numbers.

\begin{theorem}
\label{main}Let $X$ be an $r$-dimensional $(r\geq 1)$ orientable manifold
with Betti numbers $b_{i}(X;\mathbb{Z}/2).$ Suppose that $\Sigma
_{i=0}^{r}b_{i}(X;\mathbb{Z}/2)\leq 4$. When $r<n-1,$ any group action of $%
\mathrm{SL}_{n}(\mathbb{Z})$ on $X$ is trivial if the group action is by

\begin{enumerate}
\item[(i)] diffeomorphisms; or

\item[(ii)] piecewise linear (PL) homeomorphisms; or

\item[(iii)] homeomorphisms and $r\leq 4.$
\end{enumerate}
\end{theorem}

When the group action is by diffeomorphisms (piecewise linear
homeomorphisms, resp.), we always assume that $X$ is smooth (piecewise
linear, resp.). The proof of Theorem \ref{main} is to do induction on the
dimension $r$ by studying the action of involutions on $X.$ The assumptions
of various actions are to make sure that the fixed point set of an
orientation-preserving involution is a manifold and is of even codimension.

By Theorem \ref{main}, Conjecture \ref{conj} holds for actions of $\mathrm{SL%
}_{n}(\mathbb{Z})$ on manifolds $X$ of the following type:

\begin{itemize}
\item $X=S^{r},$ the sphere, or $\mathbb{R}^{r},$ the Euclidean space. This
is first proved by Bridson and Vogtmann \cite{bv}.

\item $X=S^{r_{1}}\times S^{r_{2}},$ the product of two spheres.

\item $X=\mathbb{R}^{r_{1}}\times S^{r_{2}},$ the product of a sphere and an
Euclidean space.

\item $X=\mathbb{R}^{r_{0}}\times S^{r_{1}}\times S^{r_{2}},$ the product of
two spheres and an Euclidean space.
\end{itemize}

We study matrix group actions on manifolds with few Betti numbers by
homeomorphisms. For such topological actions, we have the following.

\begin{theorem}
\label{main3}Let $X$ be an $r$-dimensional $(r\geq 1)$ homology manifold
over $\mathbb{Z}/p$ $(p$ a prime) with Betti numbers $b_{i}(X;\mathbb{Z}/p).$
Suppose that $\Sigma _{i=0}^{r}b_{i}(X;\mathbb{Z}/p)\leq 4$. We have the
following.

\begin{enumerate}
\item[(i)] Assume $p=2.$ When $r<n-3$ if $n$ is odd, or $r<n-4$ if $n$ is
even, any group action of $\mathrm{SL}_{n}(\mathbb{Z})$ on $X$ by
homeomorphisms is trivial$.$

\item[(ii)] Assume $p\ $is odd. When $r<n-2$ if $n$ is even, or $r<n-3$ if $%
n $ is odd, any group action of $\mathrm{SL}_{n}(\mathbb{Z})$ on $X$ by
homeomorphisms is trivial$.$

\item[(iii)] When $p=2$ and $r<n-3,$ any group action of $\mathrm{SL}_{n}(%
\mathbb{Z}/2)$ on $X$ by homeomorphisms is trivial$.$ When $p\ $is odd and $%
r<2n-4,$ any group action of $\mathrm{SL}_{n}(\mathbb{Z}/p)$ on $X$ by
homeomorphisms is trivial$.$
\end{enumerate}
\end{theorem}

We now study the inheritance of Conjecture \ref{conj} for covering spaces.
We obtain the following result.

\begin{theorem}
\label{main2}\bigskip Let $p:M^{\prime }\rightarrow M$ be a universal
covering of connected manifolds with the finitely generated abelian group $%
\mathbb{Z}^{k}$ as the deck transformation group. Suppose that

\begin{enumerate}
\item[(i)] any group action of $\mathrm{SL}_{n}(\mathbb{Z})$ on $M^{\prime }$
by homeomorphisms (resp. diffeomorphisms, PL homeomorphisms) is trivial$;$

\item[(ii)] $k\leq n-1.$
\end{enumerate}

Then any group action of $\mathrm{SL}_{n}(\mathbb{Z})$ on $M$ by
homeomorphisms (resp. diffeomorphisms, PL homeomorphisms) is trivial.
\end{theorem}

The proof of Theorem \ref{main2} is to lift the action on $M$ to the
covering space $M^{\prime },$ using cohomology of groups.

The linear group action of $\mathrm{SL}_{n}(\mathbb{Z})$ on the Euclidean
space $\mathbb{R}^{n}$ induces an action on the $n$-dimensional torus $T^{n}=%
\mathbb{R}^{n}/\mathbb{Z}^{n}.$ Weinberger \cite{we} shows that when $r<n,$
any group action of $\mathrm{SL}_{n}(\mathbb{Z})$ on the lower dimensional
torus $T^{r}$ by diffeomorphisms is trivial, i.e., Conjecture \ref{conj}
holds for the group $\mathrm{SL}_{n}(\mathbb{Z})$ itself and $M=T^{r}$. As
an application of Theorem \ref{main2}, we obtain a topological version of
Weinberger's result, as a special case of the following (note that
Weinberger \cite{we} also outlines a proof of the topological version).

\begin{corollary}
\label{torus} We have the following.

\begin{enumerate}
\item[(i)] When $r<n-1,$ any group action of $\mathrm{SL}_{n}(\mathbb{Z})$
on the $r$-dimensional torus $T^{r}$ by homeomorphisms is trivial$.$

\item[(ii)] Let $M=T^{r_{1}}\times S^{r_{2}}$ be the product of a torus and
a sphere of dimension $r=r_{1}+r_{2},$ or $M=T^{r_{0}}\times S^{r_{1}}\times
S^{r_{2}}$ be the product of a torus and a sphere of dimension $%
r=r_{0}+r_{1}+r_{2}.$ When $r<n-1,$ any group action of $\mathrm{SL}_{n}(%
\mathbb{Z})$ on $M$ by diffeomorphisms or PL homeomorphisms is trivial$.$

\item[(iii)] Let $M=S^{2}\times S^{2},$ or $M=T^{r_{1}}\times S^{r_{2}}$ be
the product of a torus and a sphere of dimension $r=r_{1}+r_{2}\leq 4,$ (eg. 
$M=T^{2}\times S^{2},S^{1}\times S^{3}).$ When $r<n-1,$ any group action of $%
\mathrm{SL}_{n}(\mathbb{Z})$ on $M$ by homeomorphisms is trivial$.$
\end{enumerate}
\end{corollary}

\bigskip Corollary \ref{torus} shows that Conjecture \ref{conj} holds for
the group $\mathrm{SL}_{n}(\mathbb{Z})$ itself and $M=T^{r_{1}}\times
S^{r_{2}},$ the product of a torus, or $M=T^{r_{0}}\times S^{r_{1}}\times
S^{r_{2}},$ the product of a torus and two spheres.

The article is organized as follows. In Section 2, some basic facts about
homology manifolds over sheaves are introduced. In Section 3 and Section 4,
Theorem \ref{main} and Theorem \ref{main3} are proved. In the last section,
we prove Theorem \ref{main2} and Corollary \ref{torus}.

\section{Homology manifolds}

In this section, we introduce some basic notions on generalized manifolds,
following Bredon's book \cite{b}. Let $L=\mathbb{Z}$ or $\mathbb{Z}/p.$ All
homology groups in this section are Borel-Moore homology with compact
supports and coefficients in a sheaf $\mathcal{A}$ of modules over a
principle ideal domain $L$. The homology groups of $X$ are denoted by $%
H_{\ast }^{c}(X;\mathcal{A})$ and the Alexander-Spanier cohomology groups
(with coefficients in $L$ and compact supports) are denoted by $H_{c}^{\ast
}(X;L).$ We define $\dim _{L}X=\min \{n\mid H_{c}^{n+1}(U;L)=0$ for all open 
$U\subset X\}.$ If $L=\mathbb{Z}/p,$ we write $\dim _{p}X$ for $\dim _{L}X.$
For integer $k\geq 0,$ let $\mathcal{O}_{k}$ denote the sheaf associated to
the pre-sheaf $U\longmapsto H_{k}^{c}(X,X\backslash U;L).$

\begin{definition}
An $n$-dimensional homology manifold over $L$ (denoted $n$-hm$_{L}$) is a
locally compact Hausdorff space $X$ with $\dim _{L}X<+\infty $, and $%
\mathcal{O}_{k}(X;L)=0$ for $k\neq n$ and $\mathcal{O}_{n}(X;L)$ is locally
constant with stalks isomorphic to $L$. The sheaf $\mathcal{O}_{n}$ is
called the orientation sheaf.
\end{definition}

There is a similar notion of cohomology manifold over $L$, denoted $n$-cm$%
_{L}$ (cf. \cite{b}, p.373). In this article, we assume that the homology
manifolds are second countable. The following result is generally called the
Local Smith Theorem (cf. \cite{b} Theorem 20.1, Prop 20.2, pp. 409-410).

\begin{lemma}
\label{sm}Let $p$ be a prime and $L=\mathbb{Z}/p$. The fixed point set of
any action of $\mathbb{Z}/p$ on an $n$-hm$_{L}$ is the disjoint union of
(open and closed) components each of which is an $r$-hm$_{L}$ with $r\leq n$%
. If $p$ is odd then each component of the fixed point set has even
codimension.
\end{lemma}

In order to prove Theorem \ref{main}, we need several lemmas. The $i^{th}$
modulo $p$ Betti number $b_{i}(X;\mathbb{Z}/p)$ of $X$ is defined as the
dimension of vector space $H_{c}^{i}(X;\mathbb{Z}/p).$ If the prime $p$ is
clear from the context, we simply write $b_{i}(X)$ instead of $b_{i}(X;%
\mathbb{Z}/p)$. If $\dim _{p}X=n,$ we say that $X$ satisfies the modulo $p$
Poincar\'{e} duality if $b_{i}(X;\mathbb{Z}/p)=b_{n-i}(X;\mathbb{Z}/p)$ for
all $i\geq 0.$ Note that homology manifolds satisfy Poincar\'{e} duality
between Borel-Moore homology and sheaf cohomology (\cite{b}, Theorem 9.2,
p.329), i.e. if $X$ is an $n$-hm$_{L}$ then%
\begin{equation*}
H_{c}^{n-k}(X;\mathcal{O}_{n})\cong H_{k}^{c}(X;L).
\end{equation*}

The following lemma says that an elementary $p$-group acting freely on a
manifold, whose Betti numbers satisfy $\sum_{i=0}^{n}b_{i}(X)\leq 4,$ is of
rank at most $2$ (cf. \cite{mann}, Lemma 3.1).

\begin{lemma}
\label{mann}Suppose that $X$ is a space with $\dim _{p}X=n,$ $X$ satisfies
modulo $p$ Poincar\'{e} duality, and the sum of Betti numbers $%
\sum_{i=0}^{n}b_{i}(X)\leq 4.$ If $G$ is an elementary $p$-group of rank $k$
acting freely on $X,$ then $k\leq 2.$
\end{lemma}

\begin{proof}
When the space $X$ is assumed to be compact, it is already proved in \cite%
{mann}. Actually, the statement holds for any general space, as follows. If $%
\sum_{i=0}^{n}b_{i}(X)=1,$ the space $X$ is acyclic; hence, $k=0$ (cf. \cite%
{b}, Corollary 19.8, p.144). If $\sum_{i=0}^{n}b_{i}(X)=2,$ the space $X$ is
a modulo $p$ cohomology $n$-sphere. A classical result of Smith \cite{sm}
says that $k\leq 1.$ If $\sum_{i=0}^{n}b_{i}(X)=4,$ then $X$ has the modulo $%
p$ Betti numbers of a product of two spheres and $k\leq 2$ by a result of
Heller \cite{he} (Theorem 2). If $\sum_{i=0}^{n}b_{i}(X)=3,$ similar
techniques in \cite{he} prove that $k\leq 2.$ For example, substitution $%
k=r=3$ into (3.3) of \cite{he} leads to a contradiction.
\end{proof}

\bigskip

The following result relates the Betti numbers, Euler characteristics of the
fixed point set and those of the whole space (cf. Floyd \cite{flo}).

\begin{lemma}
\label{floy}Let $p$ be a prime number and $G=\mathbb{Z}/p$ act on a
cohomological manifold $X$ with the fixed point set $F.$ Then 
\begin{equation*}
\chi (F)\equiv \chi (G)\func{mod}p,
\end{equation*}%
and 
\begin{equation*}
\sum b_{i}(F)\leq \sum b_{i}(X).
\end{equation*}
\end{lemma}

There is a relation between dimensions of fixed point set and the whole
space as follows (cf. Borel \cite{Bo}, Theorem 4.3, p182).

\begin{lemma}
\label{borel}Let $G$ be an elementary $p$-group operating on a first
countable cohomology $n$-manifold $X$ mod $p.$ Let $x\in X$ be a fixed point
of $G$ on $X$ and let $n(H)$ be the cohomology dimension mod $p$ of the
component of $x$ in the fixed point set of a subgroup $H$ of $G.$ If $%
r=n(G), $ we have 
\begin{equation*}
n-r=\sum_{H}(n(H)-r)
\end{equation*}%
where $H\ $runs through the subgroups of $G$ of index $p.$
\end{lemma}

\bigskip

The following lemma is a key step in the inductive proof of Theorem \ref%
{main} (cf. Bredon \cite{bre2}, Theorem 7.1).

\begin{lemma}
\label{Bredon}Let $G$ be a group of order $2$ operating effectively on an $n$%
-cm over $\mathbb{Z}$, with nonempty fixed points. Let $F_{0}$ be a
connected component of the fixed point set of $G$, and $r=\dim _{2}F_{0}$.
Then $n-r$ is even (respectively odd) if and only if $G$ preserves
(respectively reverses) the local orientation around some point of $F_{0}.$
\end{lemma}

\section{Smooth and piecewise linear actions}

In order to prove Theorem \ref{main}, we need several technical lemmas.

\begin{lemma}
\label{sym}Let $S_{n}$ be the permutation group of $n$ letters and $k\geq 2$
be an integer. For any $n\leq 4$ and any group homomorphism $f:\mathrm{SL}%
_{k}(\mathbb{Z}/2)\ltimes (\mathbb{Z}/2)^{k}\rightarrow S_{4},$ where $%
\mathrm{SL}_{k}(\mathbb{Z}/2)$ acts on $(\mathbb{Z}/2)^{k}$ by matrix
multiplications, there is an element $\sigma \in \mathrm{SL}_{k}(\mathbb{Z}%
/2)\ltimes (\mathbb{Z}/2)^{k}$ of order two such that the fixed point set $%
\mathrm{Fix(}f(\sigma ))\neq \emptyset .$
\end{lemma}

\begin{proof}
When $k\geq 3,$ the group $\mathrm{SL}_{k}(\mathbb{Z}/2)$ is perfect.
Therefore, $\func{Im}f$ consist only the identity element and the claim is
trivial. It is only need to prove the case when $k=2.$ For the group $%
\mathrm{SL}_{k}(\mathbb{Z}/2),$ let 
\begin{equation*}
A=%
\begin{pmatrix}
0 & 1 \\ 
1 & 0%
\end{pmatrix}%
\text{ and }B=%
\begin{pmatrix}
1 & 1 \\ 
0 & 1%
\end{pmatrix}%
.
\end{equation*}%
Note that $A^{2}=B^{2}=I.$ When $n=3,$ $f((A,(0,0))$ has order at most $2,$
for which we take $\sigma =(A,(0,0)).$ When $n=4,$ if $f((A,(0,0))$ (or $%
f(B,(0,0)))$ is not a product of disjoint transpositions, we take $\sigma
=(A,(0,0))$ (or $(B,(0,0))).$ We suppose that both $f((A,(0,0)))$ and $%
f((B,,(0,0))$ are products of two disjoint transpositions when $n=4$ or that
both $f((A,(0,0)))$ and $f((B,,(0,0))$ are transpositions when $n=2$. Then $%
f((AB,(0,0))=\mathrm{id}\in S_{n},$ since $AB$ is of order $3\ $while a $3$%
-cycle in $S_{4}$ is not a product of such transpositions. Take 
\begin{equation*}
\sigma =(I_{2},(1,0))=(AB,(0,0))(I,(0,1))(B^{-1}A^{-1},(0,0))(I_{2},(0,1)).
\end{equation*}%
It is direct that $f(\sigma )=$ $\mathrm{id}\in S_{n}$ for $n=2,4.$
\end{proof}

\bigskip\ 

For $1\leq i\neq j\leq n,$ let $e_{ij}(1)$ denote the matrix with $1$s along
the diagonal and $1$ in the $(i,j)$-th position and zeros elsewhere. Let $D<%
\mathrm{SL}_{n}(\mathbb{Z}/p)$ be the subgroup generated by $%
\{e_{1i}(1),i=2,\ldots ,n\}$.

\begin{lemma}
\label{finite conj}Let $n\geq 3.$ The subgroup $D$ is isomorphic to $(%
\mathbb{Z}/p)^{n-1}$ and any two nontrivial elements in $\mathrm{SL}_{n}(%
\mathbb{Z}/p)$ are conjugate.
\end{lemma}

\begin{proof}
It is obvious that $e_{1i}(1),i=2,\ldots ,n,$ commutes with each other and
thus $D\cong (\mathbb{Z}/p)^{n-1}$. Let $A=\prod_{i=2}^{n}e_{1i}(x_{i})\in D$
be any nontrivial element, where $(x_{2},x_{3},\ldots ,x_{n})\in (\mathbb{Z}%
/p)^{n-1}$ is a nonzero vector. After conjugacing by a permutation matrix,
we may assume $x_{2}\neq 0$ (for example, when $x_{3}\neq 0,$ the $(1,2)$-th
entry of $P_{23}AP_{23}^{-1}$ is nonzero. Here 
\begin{equation*}
P_{23}=%
\begin{pmatrix}
1 &  &  &  \\ 
&  & 1 &  \\ 
& -1 &  &  \\ 
&  &  & I_{n-3}%
\end{pmatrix}%
\end{equation*}%
). If $x_{2}\neq 0,$ we have$\ $%
\begin{equation*}
(\Pi _{i=3}^{n}e_{2i}(x_{2}^{-1}x_{i}))^{-1}\cdot A\cdot \Pi
_{i=3}^{n}e_{2i}(x_{2}^{-1}x_{i})=e_{12}(x_{2}).
\end{equation*}%
For any nonzero element $x\in \mathbb{Z}/p,$ $%
e_{23}(-x_{2}^{-1}x)e_{12}(x_{2})e_{12}(x_{2}^{-1}x)$ has $x$ as its $(1,3)$%
-th entry. After permuting with $P_{23},$ the $(1,2)$-entry of 
\begin{equation*}
P_{23}e_{23}(-x_{2}^{-1}x)e_{12}(x_{2})e_{12}(x_{2}^{-1}x)P_{23}^{-1}
\end{equation*}%
is $x.$ Therefore, any nontrivial element in $D$ is conjugate to $e_{12}(1),$
by taking $x=1.$
\end{proof}

\bigskip

We prove Theorem \ref{main} by induction on the dimension $r.$ As a
induction step, we prove the following result first.

\begin{theorem}
\label{main copy(1)}Let $X$ be a homology manifold over $\mathbb{Z}/2$ with
Betti numbers $b_{i}(X;\mathbb{Z}/2)$ of dimension $r\leq 2.$ Suppose that $%
\Sigma _{i=0}^{r}b_{i}(X;\mathbb{Z}/2)\leq 4$. When $r<n-1,$ any group
action of $\mathrm{SL}_{n}(\mathbb{Z})$ $(n\geq 3)$ on $X$ by homeomorphisms
is trivial$.$
\end{theorem}

\begin{proof}
Since $r\leq 2,$ $X$ is a topological manifold (cf. \cite{b}, 16.32, p.388).
We prove the theorem by induction on $r.$ When $r=1,$ $X$ is the disjoint
union of several copies of $S^{1}$ and $\mathbb{R}^{1}.$ Since $\mathrm{SL}%
_{n}(\mathbb{Z})$ $(n\geq 3)$ is perfect, the group action preserves each
component. It is already known that $\mathrm{SL}_{n}(\mathbb{Z})$ $(n\geq 3)$
can only act trivially on $S^{1}$ or $\mathbb{R}^{1}$ (cf. Witte \cite{wi}).

Suppose $r=2.$ Let $D=\{\mathrm{diag}(\pm 1,\pm 1,\ldots ,\pm 1)<\mathrm{SL}%
_{n}(\mathbb{Z})\}\cong (\mathbb{Z}/2)^{n-1}.$ Let $A_{i}=\mathrm{diag}%
(-1,\ldots ,-1,\ldots ,1),$ where the second $-1$ is in the $(i+1,i+1)$-th
position$.$ By Lemma \ref{mann}, the subgroup $A=\langle
A_{1},A_{2},A_{3}\rangle $ can't act freely on $X.$ There are several cases
to consider.

\begin{enumerate}
\item[\textbf{Case (1)}] $\mathrm{Fix}(A_{1}A_{2}A_{3})=\emptyset .$
\end{enumerate}

By Lemma \ref{mann}, $\mathrm{Fix}(A_{1})\neq \emptyset ,$ since all other
nontrivial elements in $A$ except $A_{1}A_{2}A_{3}$ are conjugate.

\begin{enumerate}
\item[\textbf{Case (1.1)}] $\dim _{2}\mathrm{Fix}(A_{1})=2.$
\end{enumerate}

By the invariance of domain, $\mathrm{Fix}(A_{1})=X.$ Let $N$ be the normal
subgroup generated by $A_{1}.$ Then $\mathrm{SL}_{n}(\mathbb{Z})/N\cong 
\mathrm{SL}_{n}(\mathbb{Z}/2)$ (cf. Ye \cite{ye}, Proposition 3.4). This
means that the group action of $\mathrm{SL}_{n}(\mathbb{Z})$ factors through 
$\mathrm{SL}_{n}(\mathbb{Z}/2).$ Let $e_{12}(1)\in \mathrm{SL}_{n}(\mathbb{Z}%
/2).$ Since the subgroup $D$ generated by $\{e_{1i}(1),i=2,\ldots ,n\}$ is
isomorphic to $(\mathbb{Z}/2)^{n-1}.$ Using Lemma \ref{mann} once again, $A$
can't act freely on $X$ (note that $n\geq 4$ when $r=2$). Since any two
nonzero elements in $D$ are conjugate in $\mathrm{SL}_{n}(\mathbb{Z}/2)$
(cf. Lemma \ref{finite conj}), the fixed point set $\mathrm{Fix}%
(e_{12}(1))\neq \emptyset .$ We consider two cases.

\begin{description}
\item[Case (1.1.1)] $\dim _{2}\mathrm{Fix}(e_{12}(1))=2.$
\end{description}

By the invariance of domain, we have $\mathrm{Fix}(e_{12}(1))=X.$ This means
that $e_{12}(1)$ acts trivially on $X.$ Note fact that $\mathrm{SL}_{n}(%
\mathbb{Z}/2)$ is a simple group for $n\geq 3$ and thus $e_{12}(1)$ normally
generates the whole group $\mathrm{SL}_{n}(\mathbb{Z}/2).$ Therefore, the
group action of $\mathrm{SL}_{n}(\mathbb{Z}/2)$ on $X$ is trivial.

\begin{description}
\item[Case (1.1.2)] $\dim _{2}\mathrm{Fix}(e_{12}(1))<2.$
\end{description}

Since $\mathrm{SL}_{n}(\mathbb{Z})$ $(n\geq 3)$ is perfect, the group action
on $X$ is orientation-preverving. By Lemma \ref{Bredon}, $\dim _{p}\mathrm{%
Fix}(e_{12}(1))=0.$ Thus, $r=2$ and $n\geq 4.$ This means that $\mathrm{Fix}%
(e_{12}(1))$ is a discrete set consisting of at most $4$ points, considering
Lemma \ref{floy}. Note that the centralizer $C_{\mathrm{SL}_{n}(\mathbb{Z}%
/2)}(e_{12}(1))$ leaves $\mathrm{Fix}(e_{12}(1))$ invariant and $C_{\mathrm{%
SL}_{n}(\mathbb{Z})}(e_{12}(1))$ contains a copy of $\mathrm{SL}_{n-2}(%
\mathbb{Z}/2)\ltimes (\mathbb{Z}/2)^{n-2},$ where $\mathrm{SL}_{n-2}(\mathbb{%
Z}/2)$ acts on $(\mathbb{Z}/2)^{n-2}$ by matrix multiplications. Since the
permutation group of $\mathrm{Fix}(e_{12}(1))$ is at most $S_{4},$ there is
an element $b\in \mathrm{SL}_{n-2}(\mathbb{Z}/2)\ltimes (\mathbb{Z}/2)^{n-2}$
of order 2 such that $\mathrm{Fix}(b)\cap \mathrm{Fix}(e_{12}(1))\neq
\emptyset $, according to Lemma \ref{sym}. If $\dim _{p}\mathrm{Fix}(b)=2,$
a similar argument as Case (1.1.1) shows that the group action of $\mathrm{SL%
}_{n}(\mathbb{Z})$ is trivial. Otherwise, $\dim _{p}\mathrm{Fix}(b)=0$ by %
\ref{Bredon}. Let $B=\langle e_{12}(1),b\rangle \cong (\mathbb{Z}/2)^{2}.$
Note that $\mathrm{Fix}(B)\neq \emptyset .$ Write $m=\dim _{2}(\mathrm{Fix}%
(B))$ and $n(H)=\dim _{2}(\mathrm{Fix}(H))$ for each non-trivial cyclic
subgroup $H<B.$ By Lemma \ref{borel}, 
\begin{equation*}
r-m=\sum_{H}(n(H)-m).
\end{equation*}%
The only nonzero term in the right hand side is $n(\langle e_{12}(1)b\rangle
)=\dim _{p}\mathrm{Fix}(e_{12}(1)b)=2.$ A similar argument as (1.1.1) shows
that the group action of $\mathrm{SL}_{n}(\mathbb{Z})$ is trivial.

\begin{description}
\item[Case (1.2)] $\dim _{2}\mathrm{Fix}(A_{1})<2.$
\end{description}

By Lemma \ref{Bredon}, $\dim _{2}\mathrm{Fix}(A_{1})=0.$ Considering Lemma %
\ref{floy}, $\mathrm{Fix}(A_{1})$ is a discrete set consisting of at most $4$
points. The centralizer $C_{\mathrm{SL}_{n}(\mathbb{Z})}(A_{1})$ leaves $%
\mathrm{Fix}(A_{1})$ invariant. The action of $C_{\mathrm{SL}_{n}(\mathbb{Z}%
)}(A_{1})$ restricts to be an action of $(\mathbb{Z}/2)^{3}$ generated by $%
A_{2},A_{3}$ and 
\begin{equation*}
a=%
\begin{pmatrix}
& 1 &  \\ 
-1 &  &  \\ 
&  & I_{n-2}%
\end{pmatrix}%
.
\end{equation*}%
By Lemma \ref{mann}, $\ $a nontrivial element $c\in (\mathbb{Z}/2)^{3}$ acts
trivially on $\mathrm{Fix}(A_{1}).$ If $\dim _{2}\mathrm{Fix}(c)=2,$ we can
continue the proof as Case (1.1). Suppose $\dim _{2}\mathrm{Fix}(c)=0.$ Let $%
A=\langle A_{1},c\rangle .$ Write $m=\dim _{2}(\mathrm{Fix}(A))$ and $%
n(H)=\dim _{2}(\mathrm{Fix}(H))$ for each non-trivial cyclic subgroup $H<A.$
By Lemma \ref{borel}, 
\begin{equation*}
r-m=\sum_{H}(n(H)-m),
\end{equation*}%
where $H\ $runs through the nontrivial cyclic subgroups of $G$. Note that $%
r=2$ and $m=0.$ We see that $\dim _{2}\mathrm{Fix}(A_{1}c)=2.$ By the
invariance of domain, $A_{1}c$ acts trivially on $X$. Therefore, the group
action of $\mathrm{SL}_{n}(\mathbb{Z})$ factors through $\mathrm{SL}_{n}(%
\mathbb{Z}/2).$ A similar argument as in Case (1.1) shows that the group
action of $\mathrm{SL}_{n}(\mathbb{Z})$ is trivial.

\begin{description}
\item 
\begin{enumerate}
\item[\textbf{Case (2)}] $\mathrm{Fix}(A_{1}A_{2}A_{3})\neq \emptyset .$
\end{enumerate}
\end{description}

If $\mathrm{Fix}(A_{1})\neq \emptyset ,$ we may do a similar argument as
case (1) to conclude that the action of $\mathrm{SL}_{n}(\mathbb{Z})$ is
trivial. We suppose $\mathrm{Fix}(A_{1})=\emptyset .$ There are several
cases to consider.

\begin{description}
\item[Case (2.1)] $\dim _{2}\mathrm{Fix}(A_{1}A_{2}A_{3})=2.$
\end{description}

By the invariance of domain, the action of $A_{1}A_{2}A_{3}$ is trivial. If $%
n>4,$ the group action of $\mathrm{SL}_{n}(\mathbb{Z})$ factors through $%
\mathrm{SL}_{n}(\mathbb{Z}/2)$ (cf. Ye \cite{ye}, Proposition 3.4). We could
finish the proof by a similar argument as Case (1.1). If $n=4,$ the group
action of $\mathrm{SL}_{n}(\mathbb{Z})$ factors through $\mathrm{PSL}_{n}(%
\mathbb{Z}),$ the projective linear group. Let 
\begin{equation*}
\sigma =%
\begin{pmatrix}
& 1 &  &  \\ 
-1 &  &  &  \\ 
&  &  & 1 \\ 
&  & -1 & 
\end{pmatrix}%
\in \mathrm{SL}_{4}(\mathbb{Z}).
\end{equation*}
Since the subgroup $B=\langle A_{1},A_{2},\sigma \rangle \cong (\mathbb{Z}%
/2)^{3}<\mathrm{PSL}_{n}(\mathbb{Z}),$ we have that the fixed point set $%
\mathrm{Fix}(\sigma ),\mathrm{Fix}(A_{1}\sigma ),\mathrm{Fix}%
(A_{1}A_{2}\sigma )$ or $\mathrm{Fix}(A_{2}\sigma )$ is not empty. Without
loss of generality, assume that $\mathrm{Fix}(\sigma )\neq \emptyset .$

\begin{description}
\item[Case (2.1.1)] $\dim _{2}\mathrm{Fix}(\sigma )=2.$ By the invariance of
domain, $\sigma $ acts trivially on $X.$ Note that $\sigma $ normally
generated the whole group $\mathrm{SL}_{n}(\mathbb{Z}).$ Thus, the group
action of $\mathrm{SL}_{n}(\mathbb{Z})$ is trivial.

\item[Case (2.1.2)] $\dim _{2}\mathrm{Fix}(\sigma )<2.$
\end{description}

By Lemma \ref{Bredon}, $\dim _{2}\mathrm{Fix}(\sigma )=0.$ Since $\mathrm{Fix%
}(A_{1})=\mathrm{Fix}(A_{2})=\mathrm{Fix}(A_{1}A_{2})=\emptyset $ and $%
\langle A_{1},A_{2}\rangle \cong (\mathbb{Z}/2)^{2}$ acts freely on $\mathrm{%
Fix}(\sigma ),$ the set $\mathrm{Fix}(\sigma )$ consists of $4$ points.
Moreover, $A_{1},A_{2}$ and $A_{1}A_{2}$ acts on $\mathrm{Fix}(\sigma )$ by
products of two disjoint transpositions. Let 
\begin{equation*}
\sigma _{1}=%
\begin{pmatrix}
&  & 1 &  \\ 
&  &  & 1 \\ 
1 &  &  &  \\ 
& 1 &  & 
\end{pmatrix}%
\in \mathrm{SL}_{4}(\mathbb{Z}/2).
\end{equation*}%
Since $\sigma _{1}\sigma =\sigma \sigma _{1},$ $\sigma _{1}$ preserves $%
\mathrm{Fix}(\sigma ).$ If $\mathrm{Fix}(\sigma _{1})\cap \mathrm{Fix}%
(\sigma )\neq \emptyset ,$ then $\dim _{2}\mathrm{Fix}(\sigma \sigma _{1})=2$
by Lemma \ref{borel}. A similar argument as Case (2.1.1) shows that the
group action of $\mathrm{SL}_{n}(\mathbb{Z})$ is trivial. If $\mathrm{Fix}%
(\sigma _{1})\cap \mathrm{Fix}(\sigma )=\emptyset ,$ then $\sigma _{1}$ acts
on $\mathrm{Fix}(\sigma )$ identically as one of $\{A_{1},A_{2},A_{1}A_{2}%
\}. $ Without loss of generality, assume that $\sigma _{1}|_{\mathrm{Fix}%
(\sigma )}=A_{1}|_{\mathrm{Fix}(\sigma )}.$ By Lemma \ref{borel}, $\dim _{2}%
\mathrm{Fix}(A_{1}^{-1}\sigma _{1})=2$ and a similar argument as Case
(2.1.1) shows that the group action of $\mathrm{SL}_{n}(\mathbb{Z})$ is
trivial.

\begin{description}
\item[Case (2.2)] $\dim _{2}\mathrm{Fix}(A_{1}A_{2}A_{3})<2.$
\end{description}

By Lemma \ref{Bredon}, $\dim _{p}\mathrm{Fix}(A_{1}A_{2}A_{3})=0.$ Then the $%
\mathrm{SL}_{4}(\mathbb{Z})$ lying in the upper left corner of $\mathrm{SL}%
_{n}(\mathbb{Z})$ preserves $\mathrm{Fix}(A_{1}A_{2}A_{3}).$ Since $\mathrm{%
Fix}(A_{1}A_{2}A_{3})$ consists at most 4 points and $\mathrm{SL}_{4}(%
\mathbb{Z})=[\mathrm{SL}_{4}(\mathbb{Z}),\mathrm{SL}_{4}(\mathbb{Z})],$ the
group action of $\mathrm{SL}_{4}(\mathbb{Z})$ on $\mathrm{Fix}%
(A_{1}A_{2}A_{3})$ is trivial. This is a contradiction to the fact that $%
\mathrm{Fix}(A_{1})=\emptyset .$ The proof is finished.
\end{proof}

\bigskip

We now start to prove Theorem \ref{main}. Note that for locally finite
CW-complexes, the Borel-Moore homology coincides with singular homology. The
strategy of the proof is similar to that of Theorem \ref{main copy(1)}. We
need care about the fixed point sets and the sizes of matrix groups.

\begin{proof}[Proof of Theorem \protect\ref{main}]
We prove the theorem by induction on $r.$ By Theorem \ref{main copy(1)}, we
may suppose that the statement holds for $r\leq k-1$. We now consider the
general case for $r=k$ and $k\geq 3.$ Let $A_{i}=\mathrm{diag}(-1,\ldots
,-1,\ldots ,1),$ where the second $-1$ is in the $(i+1,i+1)$-th position$.$
By Lemma \ref{mann}, the subgroup $\langle A_{1},A_{2},A_{3}\rangle $ can't
act freely on $X.$ There are several cases to consider.

\begin{enumerate}
\item[\textbf{Case (i)}] $\mathrm{Fix}(A_{1}A_{2}A_{3})=\emptyset .$
\end{enumerate}

We have $\mathrm{Fix}(A_{1})\neq \emptyset ,$ since all nontrivial element
except $A_{1}A_{2}A_{3}$ in $\langle A_{1},A_{2},A_{3}\rangle $ are
conjugate.

\begin{description}
\item[Case (i.1)] $\dim _{2}\mathrm{Fix}(A_{1})=k.$
\end{description}

The group action of $\mathrm{SL}_{n}(\mathbb{Z})$ factors through $\mathrm{SL%
}_{n}(\mathbb{Z}/2).$ Similar argument as in Case (1.1) of the proof of
Theorem \ref{main copy(1)} shows that the group action of $\mathrm{SL}_{n}(%
\mathbb{Z}/2)$ is trivial.

\begin{description}
\item[Case (i.2)] $\dim _{2}\mathrm{Fix}(A_{1})<k.$
\end{description}

When the group action of $\mathrm{SL}_{n}(\mathbb{Z})$ is by diffeomorphisms
or piecewise linear homeomorphisms, the fixed point set $\mathrm{Fix}(A_{1})$
is a manifold. Since $\mathrm{SL}_{n}(\mathbb{Z})$ $(n\geq 3)$ is perfect,
i.e. $\mathrm{SL}_{n}(\mathbb{Z})=[\mathrm{SL}_{n}(\mathbb{Z}),\mathrm{SL}%
_{n}(\mathbb{Z})],$ the group action is orientation-preserving. When the
action is by homeomorphisms and $r\leq 4,$ the fix point set $\mathrm{Fix}%
(A_{1})$ is of codimension at least 2 by Lemma \ref{Bredon}. Then $\mathrm{%
Fix}(A_{1})$ is a manifold as well (cf. \cite{b}, 16.32, p388). Therefore, $%
\mathrm{Fix}(A_{1})$ is of codimension at least 2 in any case by Lemma \ref%
{Bredon}. We will use this fact several times in the proof. Considering
Lemma \ref{floy}, $\mathrm{Fix}(A_{1})$ consists of at most $4$ components.
If $\mathrm{Fix}(A_{1})$ has 4 components, then each component is an acyclic
homology manifold over $\mathbb{Z}/2$ by considering the Betti numbers. Note
that when $k\geq 3,$ we have $n\geq 5$ and $\mathrm{SL}_{n-2}(\mathbb{Z})$
is a perfect group. The $\mathrm{SL}_{n-2}(\mathbb{Z})$ in the centralizer $%
C_{\mathrm{SL}_{n}(\mathbb{Z})}(A_{1})$ will preserve each component. By
induction, the group action of $\mathrm{SL}_{n-2}(\mathbb{Z})$ on $\mathrm{%
Fix}(A_{1})$ is trivial. Therefore, $c=\mathrm{diag}(1,1,-1,-1,\ldots ,1)$
acts trivially on $\mathrm{Fix}(A_{1}).$ Let $A=\langle A_{1},c\rangle \cong
(\mathbb{Z}/2)^{2}.$ Write $m=\dim _{2}(\mathrm{Fix}(A))$ and $n(H)=\dim
_{2}(\mathrm{Fix}(H))$ for each non-trivial cyclic subgroup $H<A.$ By Lemma %
\ref{borel}, 
\begin{equation*}
k-m=\sum_{H}(n(H)-m),
\end{equation*}%
where the sum is taken over the non-trivial cyclic subgroup of $A.$ We see
that $\dim _{2}\mathrm{Fix}(A_{1}c)=k.$ By the invariance of domain, $A_{1}c$
acts trivially on $X$. Therefore, the group action of $\mathrm{SL}_{n}(%
\mathbb{Z})$ factors through $\mathrm{SL}_{n}(\mathbb{Z}/2).$ Similar
argument as Case (1.1) in the proof of Theorem \ref{main copy(1)} shows that
the group action of $\mathrm{SL}_{n}(\mathbb{Z}/2)$ is trivial. If $\mathrm{%
Fix}(A_{1})$ has 3 components, two of the components must be acyclic by
noting that $\sum b_{i}(F)\leq 4$ (cf. Lemma \ref{floy}). The remaining is
the same as the proof of 4 components. If $\mathrm{Fix}(A_{1})$ has at most
2 components, the $\mathrm{SL}_{n-2}(\mathbb{Z})$ in the centralizer $C_{%
\mathrm{SL}_{n}(\mathbb{Z})}(A_{1})$ will preserve each component. By
induction, the group action of $\mathrm{SL}_{n-2}(\mathbb{Z})$ on $\mathrm{%
Fix}(A_{1})$ is trivial. The remaining is the same as the proof of 4
components.

\begin{enumerate}
\item[\textbf{Case (ii)}] $\mathrm{Fix}(A_{1}A_{2}A_{3})\neq \emptyset .$
\end{enumerate}

\begin{description}
\item[Case (ii.1)] $\dim _{2}$ $\mathrm{Fix}(A_{1}A_{2}A_{3})=k.$
\end{description}

By the invariance of domain, $A_{1}A_{2}A_{3}$ acts trivially on $X$. The
group factors through $\mathrm{SL}_{n}(\mathbb{Z}/2).$ A similar argument as
in Case (1.1.1) in the proof of Theorem \ref{main copy(1)} shows that the
group action of $\mathrm{SL}_{n}(\mathbb{Z}/2)$ is trivial.

\begin{description}
\item[Case (ii.2)] $\dim _{2}$ $\mathrm{Fix}(A_{1}A_{2}A_{3})<k.$
\end{description}

If $\mathrm{Fix}(A_{1})\neq \emptyset ,$ we may continue the proof as case
(i). Suppose $\mathrm{Fix}(A_{1})=\emptyset .$ Let 
\begin{equation*}
\sigma =%
\begin{pmatrix}
& 1 &  &  \\ 
-1 &  &  &  \\ 
&  &  & 1 \\ 
&  & -1 & 
\end{pmatrix}%
\in \mathrm{SL}_{4}(\mathbb{Z})
\end{equation*}%
lying in the upper left corner of $\mathrm{SL}_{n}(\mathbb{Z}).$ Since $%
A_{1}A_{2}A_{3}$ acts trivially on $\mathrm{Fix}(A_{1}A_{2}A_{3}),$ the
group action of $\mathrm{SL}_{4}(\mathbb{Z})$ factors through $\mathrm{PSL}%
_{4}(\mathbb{Z}).$ In $\mathrm{PSL}_{4}(\mathbb{Z}),$ the image of $\langle
A_{1},A_{2},\sigma \rangle $ is isomorphic to $(\mathbb{Z}/2)^{3}.$ By Lemma %
\ref{mann}, this subgroup can't act freely on $\mathrm{Fix}%
(A_{1}A_{2}A_{3}). $ Thus, 
\begin{equation*}
\mathrm{Fix}(\sigma )\cap \mathrm{Fix}(A_{1}A_{2}A_{3})\neq \emptyset .
\end{equation*}%
If $\dim _{2}\mathrm{Fix}(\sigma )=k,$ the group action of $\mathrm{SL}_{n}(%
\mathbb{Z})$ is trivial by a similar argument as Case (ii.1). If $\dim _{2}%
\mathrm{Fix}(\sigma )<k$ and 
\begin{equation*}
\dim _{2}\mathrm{Fix}(\sigma )\cap \mathrm{Fix}(A_{1}A_{2}A_{3})=\dim _{2}%
\mathrm{Fix}(A_{1}A_{2}A_{3}),
\end{equation*}%
the group action of $\mathrm{PSL}_{4}(\mathbb{Z})$ is trivial. This is a
contradiction to the fact that $\mathrm{Fix}(A_{1})=\emptyset .$ If 
\begin{equation*}
\dim _{2}\mathrm{Fix}(\sigma )\cap \mathrm{Fix}(A_{1}A_{2}A_{3})<\dim _{2}%
\mathrm{Fix}(A_{1}A_{2}A_{3}),
\end{equation*}%
then $\dim _{2}\mathrm{Fix}(A_{1}A_{2}A_{3})-\dim _{2}\mathrm{Fix}(\sigma
)\cap \mathrm{Fix}(A_{1}A_{2}A_{3})$ is at least $2.$ The group $\mathrm{SL}%
_{n-4}(\mathbb{Z})$ lying in the lower right corner of $\mathrm{SL}_{n}(%
\mathbb{Z})$ acts on $\mathrm{Fix}(\sigma )\cap \mathrm{Fix}%
(A_{1}A_{2}A_{3}),$ which is codimension at least $4$ in $X.$ If $\dim _{2}%
\mathrm{Fix}(A_{1}A_{2}A_{3})\leq 2,$ the $\mathrm{SL}_{4}(\mathbb{Z})$
lying in the upper left corner of $\mathrm{SL}_{n}(\mathbb{Z})$ acts
trivially on $\mathrm{Fix}(A_{1}A_{2}A_{3})$ by Theorem \ref{main copy(1)}.
This is impossible by the assumption that $\mathrm{Fix}(A_{1})=\emptyset $.
If $\dim _{2}\mathrm{Fix}(A_{1}A_{2}A_{3})\geq 3,$ we have $k\geq 5$ and $%
n\geq 7.$ By induction, the group action of $\mathrm{SL}_{n-4}(\mathbb{Z})$
on $\mathrm{Fix}(\sigma )\cap \mathrm{Fix}(A_{1}A_{2}A_{3})$ is trivial.
This is a contradiction to the fact that $\mathrm{Fix}(A_{1})=\emptyset ,$
since $A_{1}$ and $\mathrm{diag}(1,,\ldots ,-1,-1)$ are conjugate. The whole
proof is finished.
\end{proof}

\bigskip

\section{Topological actions}

In this section, we will study topological actions of $\mathrm{SL}_{n}(%
\mathbb{Z})$ on manifolds with few Betti numbers. In order to prove Theorem %
\ref{main3}, we need some lemmas.

The effective group actions of elementary $p$-group on homology manifolds
with few Betti numbers was already studied by Mann \cite{mann} and Mann-Su 
\cite{ms}. The compact case of the following lemma was first obtained by
Mann \cite{mann} (Theorem 3.2).

\begin{lemma}
\label{eff}We have the following.

\begin{enumerate}
\item[(i)] If $r<d-2,$ the group $(\mathbb{Z}/2)^{d}$ cannot act effectively
on a $r$-dimensional homology manifold $X$ over $\mathbb{Z}/2$ with Betti
numbers $\Sigma _{i=0}^{r}b_{i}(X;\mathbb{Z}/2)\leq 4$.

\item[(ii)] Let $p>2$ be a prime. If $r<2d-2,$ the group $(\mathbb{Z}/p)^{d} 
$ cannot act effectively on a $r$-dimensional homology manifold $X$ over $%
\mathbb{Z}/p$ with Betti numbers $\Sigma _{i=0}^{r}b_{i}(X;\mathbb{Z}/p)\leq
4$.
\end{enumerate}
\end{lemma}

\begin{proof}
The case is vacuous for $d=1.$ For $d=2$ and $p$ is an odd prime, it is not
hard to see that $(\mathbb{Z}/p)^{2}$ can not act effectively on a $1$%
-dimensional manifold $X$ with $\Sigma _{i=0}^{r}b_{i}(X;\mathbb{Z}/p)\leq 4$
(note that $1$-hm$_{\mathbb{Z}/p}$ is actually an manifold).

We assume $d\geq 3$ and prove the theorem by induction. Let $G=(\mathbb{Z}%
/p)^{d}$ be a group acting on $X.$ Choose a nontrivial element $a\in G$ such
that the fixed point set $\mathrm{Fix}(a)$ is maximal with respect to
inclusion. Denote by $G_{0}\cong (\mathbb{Z}/p)^{d-1}$ the complement of $%
\langle a\rangle $ in $G.$ By Lemma \ref{mann}, $\mathrm{Fix}(a)$ is not
empty. If $\mathrm{Fix}(a)=X,$ we are done. If $\mathrm{Fix}(a)\neq X,$ the
codimension of $\mathrm{Fix}(a)$ is at least $1$ for $p=2$ and at least $2$
for $p$ odd by Lemma \ref{sm}. By Lemma \ref{floy}, 
\begin{equation*}
\Sigma _{i=0}^{r}b_{i}(\mathrm{Fix}(a);\mathbb{Z}/p)\leq \Sigma
_{i=0}^{r}b_{i}(X;\mathbb{Z}/p)\leq 4.
\end{equation*}%
Applying induction to the action of $G_{0}$ on $\mathrm{Fix}(a),$ there is a
nontrivial element $b\in G_{0}$ fixing $\mathrm{Fix}(a)$ pointwise. By the
maximality of $\mathrm{Fix}(a),$ we have $\mathrm{Fix}(a)=\mathrm{Fix}(b).$
Let $A=\langle a,b\rangle .$ For any nontrivial element $x\in A,$ we get $%
\mathrm{Fix}(a)=\mathrm{Fix}(x).$ Write $m=\dim _{p}(\mathrm{Fix}(A))$ and $%
n(H)=\dim _{p}(\mathrm{Fix}(H))$ for each non-trivial cyclic subgroup $H<A.$
By Lemma \ref{borel}, 
\begin{equation*}
r-m=\sum_{H}(n(H)-m),
\end{equation*}%
where the sum is taken over the non-trivial cyclic subgroup of $A.$ We have
just proved $n(H)=m$ for any $H.$ Thus $r=m.$ By the invariance of domain,
we get $\mathrm{Fix}(a)=X,$ which is a contradiction.
\end{proof}

\bigskip

\begin{proof}[Proof of Theorem \protect\ref{main3}]
We prove (i) first. Let $A_{i}=\mathrm{diag}(-1,\ldots ,-1,\ldots ,1)\in 
\mathrm{SL}_{n}(\mathbb{Z}),$ where the second $-1$ is in the $(i+1,i+1)$-th
position$.$ The subgroup $\langle A_{i},i=1,2,\ldots ,n-1\rangle \cong (%
\mathbb{Z}/2)^{n-1}.$ When $n$ is odd, the center of $\mathrm{SL}_{n}(%
\mathbb{Z})$ is trivial. By Lemma \ref{eff}, when $r<n-3$ if $n$ is odd, or $%
r<n-4$ if $n$ is even, there is a noncentral order-two element in $\mathrm{SL%
}_{n}(\mathbb{Z})$ acts trivially on $X.$ Therefore, the group action of $%
\mathrm{SL}_{n}(\mathbb{Z})$ factors through $\mathrm{SL}_{n}(\mathbb{Z}/2).$
Let $e_{1i}(1)$ denote the matrix with $1s$ along the diagonal, $1$ in the $%
(1,i)$-th position and $0s$ elsewhere. We have $\langle
e_{1i}(1),i=2,3,\ldots ,n\rangle \cong (\mathbb{Z}/2)^{n-1}.$ By Lemma \ref%
{eff}, there is a nontrivial element acting trivially on $X.$ Since $\mathrm{%
SL}_{n}(\mathbb{Z}/2)$ $(n\geq 3)$ is simple, the group action of $\mathrm{SL%
}_{n}(\mathbb{Z}/2)$ and thus that of $\mathrm{SL}_{n}(\mathbb{Z})$ is
trivial.

For (ii), let 
\begin{equation*}
B=%
\begin{pmatrix}
0 & 1 \\ 
-1 & -1%
\end{pmatrix}%
.
\end{equation*}%
It is directly that $\mathrm{SL}_{n}(\mathbb{Z})$ contains $[\frac{n}{2}]$
copies of $\langle B\rangle ,$ which is isomorphic to $(\mathbb{Z}/3)^{[%
\frac{n}{2}]}$ . Here $[\frac{n}{2}]$ is the integral part of $\frac{n}{2}.$
By Lemma \ref{eff}, there is a nontrivial element in this $(\mathbb{Z}/3)^{[%
\frac{n}{2}]}$ acting trivially on $X,$ when $r<n-2$ if $n$ is even, or $%
r<n-3$ if $n$ is odd. The nontrivial element normally generates the whole
group $\mathrm{SL}_{n}(\mathbb{Z}).$ Therefore, the group action of $\mathrm{%
SL}_{n}(\mathbb{Z})$ is trivial.

We now prove (iii). For $1\leq i\neq j\leq n,$ let $e_{ij}(1)$ denote the
matrix with $1$s along the diagonal and $1$ in the $(i,j)$-th position and
zeros elsewhere. Let $D<\mathrm{SL}_{n}(\mathbb{Z}/p)$ be the subgroup
generated by $\{e_{1i}(1),i=2,\ldots ,n\}$. It is not hard to see that $D$
is isomorphic to $(\mathbb{Z}/p)^{n-1}.$ When $r<n-3$ for $p=2$ or $r<2n-4$
for odd $p,$ there is a nontrivial element $\sigma $ in $D$ acting trivially
on $X.$ However, $\mathrm{PSL}_{n}(\mathbb{Z}/p)$ is simple and thus a
noncentral element $\sigma $ normally generates the whole group. Therefore,
the group action of $\mathrm{SL}_{n}(\mathbb{Z}/p)$ is trivial.
\end{proof}

\section{Lifting group actions}

Let $\mathrm{SL}_{n}(\mathbb{Z})$ act on the Euclidean space $\mathbb{R}^{n}$
by matrix multiplications. It induces an action on the torus $T^{n}=\mathbb{R%
}^{n}/\mathbb{Z}^{n}.$ Weinberger \cite{we} proves that any smooth action of 
$\mathrm{SL}_{n}(\mathbb{Z})$ on $T^{r}$ is trivial for $r<n$. We want to
study the group action of $\mathrm{SL}_{n}(\mathbb{Z})$ on $T^{r}$ by
homeomorphisms. Before proving Theorem \ref{torus}, we prove a general
result on group actions on manifolds and covering spaces.

Let $M$ be a connected manifold and let $G$ be a subgroup the group of
homeomorphisms of $M.$ Suppose that $p:M^{\prime }\rightarrow M$ is a
universal covering of connected manifolds with deck transformation group $%
\pi $ and denote by $G^{\prime }$ all homeomorphisms of $M^{\prime }$
covering those in $G$ (cf. \cite{bre}, Theorem 9.3, page 66). The group $%
G^{\prime }$ fits into an exact sequence%
\begin{equation}
1\rightarrow \pi \rightarrow G^{\prime }\rightarrow G\rightarrow 1.  \tag{*}
\end{equation}%
When the group action is by diffeomorphisms (piecewise linear
homeomorphisms, resp.), we always assume that the manifolds are smooth
(piecewise linear, resp.).

\begin{theorem}
\label{cover}Let $G$ be a group and $p:M^{\prime }\rightarrow M$ be a
universal covering of connected manifolds with an abelian group $\pi $ as
the deck transformation group. Suppose that

\begin{enumerate}
\item[(i)] any group action of $G$ on $M^{\prime }$ by homeomorphisms (resp.
diffeomorphisms, PL homeomorphisms) is trivial;

\item[(ii)] for any surjective group homomorphism $f:G\rightarrow Q,$ the
second cohomology group $H^{2}(Q;\pi )=0,$ if there is an effective action
of $Q$ on $M.$ Here $Q\ $acts on $\pi $ through the exact sequence (*) (as $%
G $ does).
\end{enumerate}

Then any group action of $G$ on $M$ by homeomorphisms (resp.
diffeomorphisms, PL homeomorphisms) is trivial.
\end{theorem}

\begin{proof}
Let $\mathrm{Homeo}(M)$ be the group of homeomorphisms of $M$ and $%
f:G\rightarrow \mathrm{Homeo}(M)$ a group homomorphism. Denote by $Q$ the
image $\func{Im}f$ and $G^{\prime }$ the group of all liftings of elements
in $Q$ to $\mathrm{Homeo}(M^{\prime }).$ We have a short exact sequence 
\begin{equation*}
1\rightarrow \pi \rightarrow G^{\prime }\rightarrow Q\rightarrow 1.
\end{equation*}%
By (ii), $H^{2}(Q;\pi )=0.$ Therefore, this exact sequence is split and $%
G^{\prime }\cong \pi \rtimes Q.$ The group $Q$ could act on $M^{\prime }$
through $G^{\prime }.$ Then the group $G$ could act on $M^{\prime }$ through 
$Q.$ However, this group action is trivial by assumption (i). Therefore, $Q$
is the trivial group. This proves that any group action of $G$ on $M$ by
homeomorphisms is trivial. The proof for the case of diffeomorphisms and PL
homeomorphisms is similar.
\end{proof}

\begin{lemma}
\label{cs}Let $\pi $ be a finitely generated abelian group without $2$%
-torsions. For any $n\geq 3$, the second cohomology group 
\begin{equation*}
H^{2}(\mathrm{SL}_{n}(\mathbb{Z});\pi )=0\text{ and }H^{2}(\mathrm{SL}_{n}(%
\mathbb{Z}/k);\pi )=0
\end{equation*}
for any nonzero integer $k.$
\end{lemma}

\begin{proof}
By van der Kallen \cite{van}, the second homology group $H_{2}(\mathrm{SL}%
_{n}(\mathbb{Z});\mathbb{Z})=\mathbb{Z}/2$ when $n\geq 5$ and $H_{2}(\mathrm{%
SL}_{3}(\mathbb{Z});\mathbb{Z})=H_{2}(\mathrm{SL}_{4}(\mathbb{Z});\mathbb{Z}%
)=\mathbb{Z}/2\bigoplus \mathbb{Z}/2$. Since $\mathrm{SL}_{n}(\mathbb{Z})$
is perfect, $H_{1}(\mathrm{SL}_{n}(\mathbb{Z});\mathbb{Z})=0$ for any $n\geq
3.$ By universal coefficient theorem, $H^{2}(\mathrm{SL}_{n}(\mathbb{Z});\pi
)=0$ for any $n\geq 3.$ Dennis and Stein proved that $H_{2}(\mathrm{SL}_{n}(%
\mathbb{Z}/k))=\mathbb{Z}/2$, for $k\equiv 0(\mathrm{mod}4)$, while $H_{2}(%
\mathrm{SL}_{n}(\mathbb{Z}/k))=0$, otherwise (cf. \cite{ds} corollary 10.2).
By universal coefficient theorem again, $H^{2}(\mathrm{SL}_{n}(\mathbb{Z}%
/k);\pi )=0$ for any $k.$
\end{proof}

\bigskip

\begin{proof}[Proof of Theorem \protect\ref{main2}]
It suffices to check that the two conditions in Theorem \ref{cover} are
satisfied. It is obvious that (i) in Theorem \ref{cover} is same as (i) in
Theorem \ref{main2} by taking $G=\mathrm{SL}_{n}(\mathbb{Z}).$ Since any
group homomorphism $\mathrm{SL}_{n}(\mathbb{Z})\rightarrow \mathrm{SL}_{k}(%
\mathbb{Z})$ $(k<n)$ is trivial (cf. \cite{we} or \cite{ye}, Corollary
1.11), the group $\mathrm{SL}_{n}(\mathbb{Z})$ can only act on $\mathbb{Z}%
^{k}$ trivially. Denote by $\pi =\mathbb{Z}^{k}.$ Let $f:\mathrm{SL}_{n}(%
\mathbb{Z})\rightarrow Q$ be any surjective homomorphism. If $f$ is
injective, $H^{2}(\mathrm{SL}_{n}(\mathbb{Z});\pi )=0$ by Lemma \ref{cs}. If 
$f$ is not injective, the congruence subgroup property \cite{bsm} implies
that $Q$ is a quotient of $\mathrm{SL}_{n}(\mathbb{Z}/k)$ by a central
subgroup $K$ for some integer $k.$ From the Serre spectral sequence 
\begin{equation*}
H^{p}(Q;H^{q}(K;\pi ))\Longrightarrow H^{p+q}(\mathrm{SL}_{n+1}(\mathbb{Z}%
/k);\pi ),
\end{equation*}%
we have the exact sequence%
\begin{eqnarray*}
0 &\rightarrow &H^{1}(Q;\pi )\rightarrow H^{1}(\mathrm{SL}_{n}(\mathbb{Z}%
/k);\pi )\rightarrow H^{0}(Q;H^{1}(K;\pi )) \\
&\rightarrow &H^{2}(Q;\pi )\rightarrow H^{2}(\mathrm{SL}_{n}(\mathbb{Z}%
/k);\pi ).
\end{eqnarray*}%
This implies that $H^{2}(Q;\pi )\cong H^{1}(K;\pi )=0,$ by Lemma \ref{cs}.
Therefore, condition (ii) of Theorem \ref{cover} is satisfied and any group
action of $\mathrm{SL}_{n}(\mathbb{Z})$ on $M$ is trivial.
\end{proof}

\bigskip\ 

\begin{proof}[Proof of Corollary \protect\ref{torus}]
Denote by $M^{\prime }=\mathbb{R}^{r}$, $\mathbb{R}^{r_{1}}\times S^{r_{2}}$
or $\mathbb{R}^{r_{0}}\times S^{r_{1}}\times S^{r_{2}},$ a universal cover
of $M.$ Bridson and Vogtmann \cite{bv} prove that when $r<n,$ any group
action of $\mathrm{SL}_{n}(\mathbb{Z})$ on $\mathbb{R}^{r}$ by homeomorphism
is trivial. Theorem \ref{main2} implies (i). When $r_{1}+r_{2}<n-1,$ any
group action of $\mathrm{SL}_{n}(\mathbb{Z})$ on $\mathbb{R}^{r_{1}}\times
S^{r_{2}},$ or $\mathbb{R}^{r_{0}}\times S^{r_{1}}\times S^{r_{2}}$ is
trivial by Theorem \ref{main}. Therefore, the statements (ii) and (iii)
follows Theorem \ref{main2} directly.
\end{proof}

\bigskip

Department of Mathematical Sciences, Xi'an Jiaotong-Liverpool University,
111 Ren Ai Road, Suzhou, Jiangsu 215123, China.

E-mail: Shengkui.Ye@xjtlu.edu.cn

\end{document}